\documentclass{amsart}

\usepackage[paperheight=11in, 
    paperwidth=8.5in,
    outer=1.2in,
    inner=1.2in,
    bottom=1.5in,
    top=1.5in]{geometry}

\usepackage{amsmath, amsxtra, amsthm, amsfonts, amssymb}
\usepackage{mathrsfs}
\usepackage{graphicx}
\usepackage{url}
\usepackage{color}
\usepackage{bbm}
\usepackage{tikz-cd}
\usetikzlibrary{matrix,shapes.geometric,calc,backgrounds}

\usepackage[T1]{fontenc}
\usepackage{enumitem}
\usepackage{pifont}

 \usepackage[hidelinks,backref=page]{hyperref} 
 \hypersetup{
    colorlinks,
    citecolor=magenta,
    filecolor=magenta,
    linkcolor=blue,
    urlcolor=black
}

\usepackage[capitalise]{cleveref}

\usepackage[alphabetic]{amsrefs}
\usepackage{mathrsfs}

\renewcommand{\emptyset}{\varnothing}
\newcommand{\CC}{\mathbb{C}}
\newcommand{\ZZ}{\mathbb{Z}}
\newcommand{\PP}{\mathbb{P}}
\newcommand{\RR}{\mathbb{R}}

\newcommand{\Mbar}{\overline{M}}
\newcommand{\M}{\mathrm{M}}
\newcommand{\be}{\mathbf{e}}

\RequirePackage{xspace}
\RequirePackage{etoolbox}
\RequirePackage{varwidth}
\RequirePackage{enumitem}
\RequirePackage{tensor}
\RequirePackage{mathtools}
\RequirePackage{longtable}
\RequirePackage{multirow}

\theoremstyle{definition}
\newtheorem{theorem}{Theorem}[section]
\newtheorem{corollary}[theorem]{Corollary}
\newtheorem{lemma}[theorem]{Lemma}
\newtheorem{proposition}[theorem]{Proposition}
\newtheorem{definition}[theorem]{Definition}

\newtheorem{remark}[theorem]{Remark}
\newtheorem{observation}[theorem]{Observation}
\newtheorem{maintheorem}{Theorem}

\title{Kapranov degrees}
\author{Joshua Brakensiek, Christopher Eur, Matt Larson, Shiyue Li}

\address{University of California, Berkeley, CA. USA.}
\email{\href{mailto:josh.brakensiek@berkeley.edu}{\nolinkurl{josh.brakensiek@berkeley.edu}}}

\address{Carnegie Mellon University. Pittsburgh, PA. USA.}
\email{\href{mailto:ceur@andrew.cmu.edu}{\nolinkurl{ceur@andrew.cmu.edu}}}

\address{Institute for Advanced Study and Princeton University. Princeton, NJ. USA.}
\email{\href{mailto:mattlarson@princeton.edu}{\nolinkurl{mattlarson@princeton.edu}}}

\address{Institute for Advanced Study.  Princeton, NJ. USA.}
\email{\href{mailto:shiyue_li@ias.edu}{\nolinkurl{shiyue_li@ias.edu}}}

\begin{document}

\maketitle

\begin{abstract}
The moduli space of stable rational curves with marked points has two distinguished families of maps: the forgetful maps, given by forgetting some of the markings, and the Kapranov maps, given by complete linear series of $\psi$-classes.
The collection of all these maps embeds the moduli space into a product of projective spaces.
We call the multidegrees of this embedding ``Kapranov degrees,'' which include as special cases the work of Witten, Silversmith, Gallet--Grasegger--Schicho, Castravet--Tevelev, Postnikov, Cavalieri--Gillespie--Monin, and Gillespie--Griffin--Levinson.
We establish, in terms of a combinatorial matching condition, upper bounds for Kapranov degrees and a characterization of their positivity.  The positivity characterization answers a question of Silversmith and gives a new proof of Laman's theorem characterizing generically rigid graphs in the plane.
We achieve this by proving a recursive formula for Kapranov degrees and by using tools from the theory of error correcting codes.
\end{abstract}

\section{Introduction}

For a finite set $S$ with cardinality $|S|\geq 3$, let $\Mbar_{0,S}$ be the Deligne--Mumford--Knudsen moduli space of stable rational curves with marked points labeled by $S$.  We write $\Mbar_{0,n}$ for $\Mbar_{0,[n]}$, where $[n] \coloneqq \{1, \dots, n\}$ for an integer $n\geq 3$.
For a pair $(S,i)$ consisting of a subset $S\subseteq [n]$ with $|S|\geq 3$ and an element $i\in S$, let
\begin{align*}
f_S & \colon \Mbar_{0,n} \to \Mbar_{0,S} \ \text{ be the forgetful map induced by forgetting the markings not in $S$, and}\\
|\psi_i| & \colon \Mbar_{0,S} \to \PP^{|S|-3} \ \text{ be the \emph{Kapranov map} induced by the complete linear series of the \emph{psi-class} $\psi_i$},
\end{align*}
where $\psi_i$ is the divisor class corresponding to the base-point-free line bundle whose fiber over a stable marked curve $C\in \Mbar_{0,S}$ is the cotangent line at the $i$-th marked point of $C$ \cite{Kapranov}.
See {\cite[Section 2]{GGL3}} for a convenient summary of relevant facts about these maps.

\begin{definition}
For such a pair $(S,i)$, we define $X_{S,i}$ to be the divisor class on $\Mbar_{0,n}$ by $X_{S,i} \coloneqq f_S^*\psi_i$.
\end{definition}

These divisor classes were first considered in \cites{singh,EHKR}, where they were used to give a basis for the cohomology ring of $\Mbar_{0,n}$. 

\begin{definition}
For a collection $(S_1,i_1), \dots, (S_{n-3}, i_{n-3})$ of $n-3$ such pairs, we define its \emph{Kapranov degree} as the intersection number
\[
\int_{\Mbar_{0,n}} X_{S_1, i_1} \dotsb X_{S_{n-3}, i_{n-3}},
\]
where $\int_{\Mbar_{0,n}} \colon H^{\ast}(\Mbar_{0, n}) \to \ZZ$ is the degree map on the cohomology ring of $\Mbar_{0,n}$.
\end{definition}

In other words, Kapranov degrees are exactly the coefficients of the multidegree of the embedding
\[
\Mbar_{0,n} \hookrightarrow \prod_{\substack{(S,i) \text{ such that}\\ i\in S\subseteq [n] \text{ and } |S|\geq 3}} \PP^{|S|-3}
\]
of $\Mbar_{0,n}$ into a product of projective spaces obtained by collecting together all compositions of forgetful maps and Kapranov maps.
Much of the enormous amount of work on intersection theory on $\Mbar_{0,n}$ can be viewed as special cases of Kapranov degrees:
\begin{enumerate}[label = (\roman*)]

\item \label{ref:witten} When $S_j = [n]$ for all $j=1, \dotsc, n-3$, Kapranov degrees in this case are intersection numbers of $\psi$-classes, for which Witten \cite{Witten} showed that $\int_{\Mbar_{0,n}} \psi_1^{a_1} \dotsb \psi_n^{a_n} = \binom{n-3}{a_1, \dotsc, a_{n}}$.
This is the first case of Witten's conjecture on the intersection numbers of $\psi$-classes on $\Mbar_{g,n}$.

\item When $|S_j| = 4$ for all $j$, Kapranov degrees in this case were studied by Silversmith \cite{Silversmith} under the name of \emph{cross-ratio degrees}, because the divisor class $X_{S,i}$ when $|S|=4$ is the pullback of the hyperplane class under the map $\Mbar_{0,n} \to \mathbb{P}^1$ given by taking the cross-ratio of the four points marked by $S$.  That is, the Kapranov degree in this case can be interpreted as the number of ways to choose $n$ marked points on $\mathbb{P}^1$ with $n-3$ prescribed cross-ratios.

\item\label{ref:lamansphere} When further $n = 2v$, and the size 4 subsets $S_j$ are obtained from a graph $G$ with vertex set $[v]$ and $2v - 3$ edges by assigning to each edge $(a,b)$ the subset $\{a,b,a+v,b+v\}$, the cross-ratio degree counts the number of (complex) realizations of $G$ on the $2$-dimensional sphere up to isometries for any fixed generic edge lengths of $G$ \cite{LamanSphere}.

\item When $\{n-1, n\} \subset S_j$ and $i_j \in \{n-1, n\}$ for all $j$, Kapranov degrees are certain mixed volumes of simplices, for which combinatorial formulas were given in special cases in \cite{Pos09} and \cite{EFLS} in terms of matching conditions.  See Remark~\ref{rem:simplices} for a detailed discussion.

\item When $S = [i]$, the divisor class $X_{[i],i}$ is also called an \emph{$\omega$-class}, and the intersection numbers of $\omega$-classes and $\psi$-classes were studied in \cites{CavlieriGillespieMonin,GGL2,GGL3} in terms of combinatorial structures such as parking functions, tournaments, and slide rules.
\item When $i_j = n$ for all $j$, the Kapranov degrees were computed in \cite[Theorem 3.2]{CTHypertrees}. In this case, the Kapranov degree is either $0$ or $1$.

\end{enumerate}

The more general framework that we consider here allows access to additional tools to study these special cases.
Let us begin by noting some immediate properties of Kapranov degrees. 

\begin{proposition}\label{prop:easy}
Kapranov degrees satisfy the following.
\begin{enumerate}
\item (Permutation invariance) $\int_{\Mbar_{0,n}} X_{S_1, i_1} \dotsb X_{S_{n-3}, i_{n-3}} = \int_{\Mbar_{0,n}} X_{\sigma(S_1), \sigma(i_1)} \dotsb X_{\sigma(S_{n-3}), \sigma(i_{n-3})}$ for any permutation $\sigma$ of $[n]$. 

\item (Nonnegativity) $\int_{\Mbar_{0,n}} X_{S_1, i_1} \dotsb X_{S_{n-3}, i_{n-3}} \ge 0$. 

\item (Monotonicity) 
$\int_{\Mbar_{0,n}} X_{S_1, i_1} \dotsb X_{S_{n-3}, i_{n-3}} \le \int_{\Mbar_{0,n}} X_{S_1 \cup T_1, i_1} \dotsb X_{S_{n-3} \cup T_{n-3}, i_{n-3}}$ for any subsets $T_1, \ldots, T_{n-3} \subseteq [n]$. 

\end{enumerate}
\end{proposition}

\begin{proof}
Permutation invariance is clear.  Nonnegativity follows from the fact that the divisor classes $X_{S,i}$ are nef, as they are pullbacks of $\psi$-classes, which are base-point-free.  Monotonicity follows by further noting that $X_{S\cup T,i} - X_{S,i}$ is an effective divisor class (see Lemma~\ref{lem:boundary}).
\end{proof}

Our first result gives a characterization of the positivity of Kapranov degrees.

\begin{maintheorem}\label{thm:cerberus}
Let us say that $S_1, \dotsc, S_{n-3}$ satisfies the \emph{Cerberus condition} if $|\bigcup_{j \in J} S_j | \ge  |J|+3$ for all nonempty $J \subseteq [n-3]$.
The Kapranov degree $\int_{\Mbar_{0,n}} X_{S_1, i_1} \dotsb X_{S_{n-3}, i_{n-3}}$ is positive if and only if $S_1, \dotsc, S_{n-3}$ satisfies the Cerberus condition.
\end{maintheorem}

Theorem~\ref{thm:cerberus} affirmatively answers \cite[Question 4.2]{Silversmith}.
When combined with \cite{LamanSphere}, Theorem~\ref{thm:cerberus} gives a new proof of Laman's theorem \cites{PollaczekGeiringer,Laman} concerning minimally rigid graphs in the plane; see Corollary~\ref{cor:laman}.
When combined with Proposition~\ref{prop:easy}(3) and the observation that $\int_{\Mbar_{0,n}} \psi_n^{n-3} = 1$, Theorem~\ref{thm:cerberus} also recovers \cite[Theorem 3.2]{CTHypertrees}.

The ``only if'' part of Theorem~\ref{thm:cerberus} is straightforward: if the Cerberus condition is violated, then $|\bigcup_{j \in J} S_j | <  |J| + 3$ for some nonempty $J\subseteq [n-3]$, 
and the product $\prod_{j \in J} X_{S_j, i_j}$ is zero because it is pulled back from $\Mbar_{0, S}$ where $S = \bigcup_{j \in J} S_j$, which has dimension strictly less than $|J|$.
We prove the ``if'' part in Section~\ref{sec:non-vanishing} by reducing to the case of cross-ratio degrees, in which case the positivity is equivalent to the algebraic independence of the rational functions corresponding to the cross-ratios.  We then verify the linear independence of the differentials of these rational functions using the GM-MDS theorem~\cites{Lovett,YildizHassibi} from the theory of error correcting codes.

\begin{remark}
After the completion of this paper, we learned that Theorem~\ref{thm:cerberus} was claimed in Theorem 4 of the preprint \cite{BrunoMella}. The proof there relies on an unjustified claim in the last paragraph of the proof of Proposition 4.1 that a certain map has maximal rank. The key case of cross-ratio degrees in Theorem~\ref{thm:cerberus} was proved in \cite[Theorem 3.3]{JordanKaszanitzky} by a similar technique to ours. 
\end{remark}

\medskip
Our second result is a recursive formula for Kapranov degrees, described as follows.
Let $\{\star\}$ be a singleton disjoint from $[n]$.
Given a partition $[n] = P\sqcup Q$ with $|P|\geq 2$ and $|Q| \geq 2$, and a pair $(S, i)$ with $S \subseteq [n]$, $|S|\geq 3$, and $i \in S$, we define an element $\operatorname{res}_P(S,i)$ in the cohomology of $\Mbar_{0,P\cup \star}$ by 
\[
\operatorname{res}_P(S,i) = \begin{cases}
X_{S,i} & \text{if $S\subseteq P$}\\
X_{(S \setminus i)\cup \star, \star} & \text{if $S\cap P = S\setminus i$}\\
X_{(S\cap P)\cup \star, i} & \text{if $S\cap P\supseteq \{i,j\}$ for some $i\neq j \in S$}\\
1 & \text{otherwise, i.e.,\ if $S\cap P = \{i\}$ or $S\cap P \subseteq S\setminus \{i,j\}$ for some $i\neq j \in S$.}
\end{cases}
\]
An illustration of these cases in Figure~\ref{fig:cases} may be helpful for visualizing the above conditions.
We similarly define $\operatorname{res}_Q(S,i)$ by replacing $P$ with $Q$ in the above.
From the definition, it follows that $\operatorname{res}_P(S,i) = 1$ if and only if $\operatorname{res}_Q(S,i) \neq 1$ and vice versa.

\begin{maintheorem}\label{thm:recursion}
Kapranov degrees satisfy the following recursion: for an element $a \in S_1 \setminus i_1$, 
\begin{align*}
\int_{\Mbar_{0,n}} X_{S_1, i_1} \dotsm X_{S_{n-3}, i_{n-3}}  = &\int_{\Mbar_{0,n}} X_{S_1\setminus a, i_1} X_{S_2, i_2} \dotsm X_{S_{n-3}, i_{n-3}}\\
& +\sum_{\substack{[n] = P\sqcup Q\\ \{a, i_1\} \subseteq P\\ S_1 \setminus \{a,i_1\} \subseteq Q}} \left( \int_{\Mbar_{0, P \cup \star}} \prod_{j=2}^{n-3} \operatorname{res}_P(S_j,i_j) \right) \cdot \left( \int_{\Mbar_{0, Q \cup \star}} \prod_{j=2}^{n-3} \operatorname{res}_Q(S_j,i_j)\right).
\end{align*}
\end{maintheorem}

When $|S_j| = 4$ for all $j \in [n-3]$, i.e., the case of cross-ratio degrees, \cref{thm:recursion} was proven in \cite{LamanSphere}. It was also proven in \cite[Lemma 3.11]{Goldner} using tropical geometry (see \cite{Silversmith}). We prove Theorem~\ref{thm:recursion} in Section~\ref{sec:recursion} by writing $X_{S, i} - X_{S \setminus a, i}$ as a sum of boundary divisors and using that the boundary divisors are products of smaller $\Mbar_{0,n}$'s. 
Theorem~\ref{thm:recursion} gives a useful algorithm for computing Kapranov degrees.\footnote{An implementation in SageMath is available at \url{https://github.com/MattLarson2399/KapranovDegrees}.}
For the case of cross-ratio degrees, the recursion further admits a combinatorial interpretation in terms of trivalent trees (Corollary~\ref{cor:comb}).

\medskip
Using Theorem~\ref{thm:recursion}, we deduce the following upper bound for Kapranov degrees.

\begin{maintheorem}\label{thm:upperbound}
For any fixed size 3 subset $\{p,q,r\}\subseteq [n]$, we define a \emph{matching} for $S_1, \dotsc, S_{n-3}$ to be a bijection $m \colon [n-3] \to [n]\setminus \{p,q,r\}$ such that $m(j) \in S_{j}$ for all $j\in [n-3]$.  We define the \emph{weight} of a matching $m$ to be $w(m) = \prod_{j} (|S_j| - 3)$, where the product is taken over $j \in [n-3]$ such that $m(j) = i_j$. The empty product is set to be $1$.
Then, we have that
\[
\int_{\Mbar_{0,n}} X_{S_1, i_1} \dotsb X_{S_{n-3}, i_{n-3}} \le \sum_{m \text{ a matching}} w(m).
\]
\end{maintheorem}

\medskip
In the case of cross-ratio degrees, the upper bound in Theorem~\ref{thm:upperbound} was proven in \cite{Silversmith} using Gromov--Witten theory.
We prove Theorem~\ref{thm:upperbound} in Section~\ref{sec:upperbound} by comparing the Kapranov degree with the intersection number of the pushforwards of the $X_{S,i}$'s under the map $f \colon \Mbar_{0,n} \to (\mathbb{P}^1)^{n-3}$ given by taking the cross-ratio of the points marked by $\{j, p, q, r\}$ for $j \in [n] \setminus \{p, q, r\}$.
Theorem~\ref{thm:recursion} is used for computing the pushforwards $f_*X_{S,i}$.

\medskip
Lastly, we give an explicit formula for the special case of Kapranov degrees where all the $S_j$ share a subset of size $3$, say $\{n-2, n-1, n\}$, and each $i_j$ lies in $\{n-2, n-1, n\}$.
For a pair $(T,j)$ consisting of $T \subseteq [n-3]$ and $j \in \{n-2, n-1, n\}$, let $Y_{T, j} = X_{T \cup \{n-2, n-1, n\}, j}$.
For $n-3$ such pairs $(T_1, j_1), \dotsc, (T_{n-3}, j_{n-3})$, we define a \emph{$3$-transversal} to be a map $t \colon [n-3] \to \{n-2, n-1, n\}$ such that there is a bijection $m \colon [n-3] \to [n-3]$ satisfying $m(i) \in T_{i}$ and $ t(m(i))= j_{i}$ for all $i \in [n-3]$.

\begin{maintheorem}\label{thm:3transversal}
The Kapranov degree $\int_{\Mbar_{0,n}} Y_{T_1, j_1} \dotsb Y_{T_{n-3}, j_{n-3}}$ is equal to the number of $3$-transversals of $(T_1, j_1), \dotsc, (T_{n-3}, j_{n-3})$. 
\end{maintheorem}

When further $j_i \in \{n-1, n\}$ for all $i\in [n-3]$, this follows from \cite[Theorem B]{EFLS}; see Remark~\ref{rem:simplices}.
We prove Theorem~\ref{thm:3transversal} by repeatedly applying Theorem~\ref{thm:recursion}.

\smallskip
Given the utility of Theorem~\ref{thm:recursion} for proving Theorems~\ref{thm:upperbound} and \ref{thm:3transversal}, one may attempt to deduce the positivity criterion for Kapranov degrees (Theorem~\ref{thm:cerberus}) from the recursion (Theorem~\ref{thm:recursion}).
We do not know how to do so; see Remark~\ref{rem:comb} for a discussion.

\subsection*{Open questions}
We conclude with some open problems:
\begin{enumerate}
\item When are Kapranov degrees $1$?
We don't have a conjecture even in the case of cross-ratio degrees. A special case of this problem arose in \cite{LamCR}.

\item It would be interesting to study the $K$-theoretic analogue of Kapranov degrees. A generalization to $K$-theory of the result in \cite{Witten} was given in \cites{Pandharipande,Lee}, and a generalization of \cite[Theorem 3.2]{CTHypertrees} was given in \cite[Section 9]{LLPP}.

\item To what extent can these results be generalized to positive genus? If $g = 1$, the analogue of Theorem~\ref{thm:cerberus} does not hold: $\int_{\Mbar_{1,2}} f_1^* \psi_1 f_2^* \psi_2 = 0$.
\end{enumerate}

\subsection*{A note on ground field characteristic}
We assume the ground field has characteristic zero in the proof of Theorem~\ref{thm:upperbound} when making Observation~\ref{obs:divisors}, but our results are valid for $\Mbar_{0,n}$ in arbitrary characteristic because the Chow ring of $\Mbar_{0,n}$ is independent of the ground field. See also Remark~\ref{rem:char}.

\subsection*{Acknowledgements}
We thank Renzo Cavalieri, June Huh, Noah Kravitz, Massimiliano Mella, Rob Silversmith, Hunter Spink, and Ravi Vakil for helpful conversations.
We also thank the referees for helpful suggestions.
The first author was supported by a Microsoft Research PhD Fellowship.
The second author is supported by US National Science Foundation (DMS-2001854 and DMS-2246518).
The third author is supported by an NDSEG fellowship.
The fourth author was supported by a Coline M. Makepeace Fellowship at Brown University Graduate School and NSF DMS-2053221.

\section{Positivity}\label{sec:non-vanishing}
In this section, we prove Theorem~\ref{thm:cerberus}.
Recall that a collection of subsets $S_1, \dotsc, S_{n-3}$ of $[n]$ satisfies the \emph{Cerberus condition} if
\begin{equation}\label{eq:Cerberus}\tag{Cerberus}
\bigg|\bigcup_{j \in J} S_j \bigg| \ge |J| + 3 \quad\text{for all nonempty $J \subseteq [n]$.}
\end{equation}
We noted in the introduction that the Kapranov degree is zero if the Cerberus condition is not satisfied.
For the converse, i.e., the ``if'' part of Theorem~\ref{thm:cerberus}, we reduce to the case of cross-ratio degrees by combining Proposition~\ref{prop:easy}(3) with a combinatorial argument (Proposition~\ref{prop:hallreduce}).
We then analyze the cross-ratio degree case using the GM-MDS theorem (Theorem~\ref{thm:gm-mds}) from the theory of error correcting codes.  We conclude with an application to graph rigidity (Corollary~\ref{cor:laman}).

\subsection{Reduction to the case of cross-ratio degrees}\label{ssec:hall}

\begin{proposition}\label{prop:hallreduce}
Let $S_1, \dotsc, S_{n-3} \subset [n-3]$, and choose $i_j \in S_j$. If $S_1, \dotsc, S_{n-3}$ satisfies the Cerberus condition, then we can find $S_j' \subset S_j$ with $i_j \in S_j'$ and $|S_j'| = 4$ such that $S_1', \dotsc, S_{n-3}'$ satisfies the Cerberus condition. 
\end{proposition}

Proposition~\ref{prop:hallreduce} is closely related to the ``generalized Hall's marriage theorem'' results in \cites{DauSongYuen,BGM}.
We prove Proposition~\ref{prop:hallreduce} using matroid theory. See \cite{Oxley} for undefined matroid terminology.

\smallskip
We begin by recalling the Hall--Rado theorem, as stated in \cite[Theorem 11.2.2]{Oxley}.

\begin{proposition}\label{prop:hallrado}
Let $\M$ be a matroid with rank function $\operatorname{rk}_\M$, and let $A_1, \dotsc, A_\ell$ be subsets of the ground set of $\M$. Then, there is an independent set $\{a_1, \dotsc, a_\ell\}$ of $\M$ satisfying $a_j \in A_j$ for all $j\in [\ell]$ if and only if
\begin{equation}\label{eq:HR}\tag{Hall--Rado}
\operatorname{rk}_\M\left(  \textstyle \bigcup_{j \in J} A_j \right) \displaystyle \ge |J| \quad\text{for all $J\subseteq [\ell]$}.
\end{equation}
In particular, such a subset $\{a_1, \dotsc, a_\ell\}$ exists if and only if it exists after replacing each $A_j$ with its closure in $\M$.
\end{proposition}

We will apply the Hall--Rado theorem to the matroid known as the \emph{third Dilworth truncation} of the boolean matroid, denoted $D_3(U_{n,n})$.  See, for example, \cite{Whittle}.
Its ground set is the collection $\binom{[n]}{4}$ of all subsets of $[n]$ of size $4$, and its set of bases consists of collections of $n-3$ subsets of size $4$ that satisfy the Cerberus condition.
In \cite[Theorem 3.2]{Athanasiadis}, the flats of $D_3(U_{n,n})$ are calculated to be as follows. They are in bijection with anti-chains $\{T_1, \dotsc, T_k\}$ in the poset of subsets of $[n]$ that satisfy the conditions
$$|T_i| \ge 4 \quad\text{and}\quad |  \textstyle \bigcup_{i \in I}\displaystyle T_i| > 3 + \sum_{i \in I} (|T_i| - 3) \quad\text{for all $I \subseteq [k] $ with $|I| \ge 2$}.$$
The flat of $D_3(U_{n,n})$ corresponding to $\{T_1, \dotsc, T_k\}$ is the set
\[
F_{T_1,\dotsc, T_k} = \left\{S \in  \textstyle \binom{[n]}{4} : \text{$S\subseteq T_i$ for some $i\in [k]$}\right\},
\]
and its rank in $D_3(U_{n,n})$ is $\sum_{i=1}^{k} (|T_i| - 3)$.
The flat $F_{T_1, \dotsc, T_k}$ is connected if and only if $k = 1$.
The poset relation on the lattice of flats of $D_3(U_{n,n})$ is by refinement: $F_{T_1, \dotsc, T_k} \le F_{T'_1, \dotsc, T_{k'}'}$ if each $T_i$ is contained in some $T_{i'}'$.

\smallskip

\begin{proof}[Proof of Proposition~\ref{prop:hallreduce}]
For a pair $(S, i)$, let $T_{S, i}$ be the set $\{S' \in \binom{[n]}{4} : i\in S' \subseteq S\}$.
We claim that because $S_1, \dotsc, S_{n-3}$ satisfies Cerberus, the collection $T_{S_1, i_1}, \dotsc, T_{S_{n-3}, i_{n-3}}$ of subsets of $\binom{[n]}{4}$ satisfies \eqref{eq:HR} with $\M = D_3(U_{n,n})$.  The result then follows by applying Proposition~\ref{prop:hallrado} with $A_j = T_{S_j,i_j}$.

For the claim, we first note that the closure of $T_{S, i}$ in $D_3(U_{n,n})$ is the flat $F_S$ corresponding to $\{S \}$ because $F_S$ contains $T_{S,i}$ and they both have rank $|S| - 3$.
We now need show that the collection of flats $F_{S_1}, \dotsc, F_{S_{n-3}}$ of $D_3(U_{n,n})$ satisfies \eqref{eq:HR}.
Suppose on the contrary that it fails \eqref{eq:HR}, and take a minimal $J\subseteq [n-3]$ witnessing the failure.
For each $j\in J$, because the flat $F_{S_j}$ is connected, it is contained in a connected component of the join $\bigvee_{j\in J} F_{S_j}$.
Since ranks add over connected components, the minimality of $J$ implies that $\bigvee_{j\in J} F_{S_j}$ is a connected flat of rank less than $|J|$ in $D_3(U_{n,n})$.
But the smallest connected flat containing $F_{S_j}$ for all $j\in J$ is the flat $F_{\bigcup_{j\in J} S_j}$, which has rank $|\bigcup_{j\in J} S_j| - 3 \geq |J|$ because of the Cerberus condition. 
\end{proof}

\subsection{Positivity for cross-ratio degrees}\label{ssec:crossratio}

\begin{proof}[Proof of Theorem~\ref{thm:cerberus}]
By Proposition~\ref{prop:hallreduce} and the monotonicity property Proposition~\ref{prop:easy}(3), it suffices to show the theorem in the case of cross-ratio degrees, which is Proposition~\ref{prop:crossratiodegcase} established below.
\end{proof}

\begin{proposition}\label{prop:crossratiodegcase}
Let $S_1, \dotsc, S_{n-3} \subseteq [n]$, each with $|S_j| = 4$, satisfying the Cerberus condition.  Then, we have
$$\int_{\Mbar_{0,n}} X_{S_1} \dotsb X_{S_{n-3}} > 0.$$
\end{proposition}
Recall that the cross-ratio of points $x_1, x_2, x_3, x_4$ in $\mathbb{P}^1$ is 
\[
\operatorname{cr}_{\{1,2,3,4\}} = \frac{(x_1 - x_3)(x_2-x_4)}{(x_1-x_4)(x_2-x_3)}.
\] 
Note that $\int_{\Mbar_{0,n}} X_{S_1} \dotsb X_{S_{n-3}}$ is the degree of the map $\Mbar_{0,n} \to (\mathbb{P}^1)^{n-3}$ induced by taking the cross-ratio of the points marked by $S_1, \dotsc, S_{n-3}$, which is positive exactly when the map is dominant.
In other words, denoting by $\operatorname{cr}_{S_i} \in K(\Mbar_{0,n})$ the rational function corresponding to the cross-ratio of the points marked by $S_i$, we have 
$$\int_{\Mbar_{0,n}} X_{S_1} \dotsb X_{S_{n-3}} > 0 \quad\text{if and only if}\quad \operatorname{cr}_{S_1}, \dotsc, \operatorname{cr}_{S_{n-3}} \text{ are algebraically independent}.$$
In order to check this, it suffices to show that
$$d\operatorname{cr}_{S_1}, \dotsc, d\operatorname{cr}_{S_{n-3}} \text{ are linearly independent in } \Omega_{K(\Mbar_{0,n})/K},$$
where $\Omega_{K(\Mbar_{0,n})/K}$ is the module of differentials of the function field $K(\Mbar_{0,n})$ over the ground field $K$. The quotient map $(\mathbb{P}^1)^{n} \dashrightarrow \Mbar_{0,n}$ induces an injection 
$$\Omega_{K(\Mbar_{0,n})/K} \otimes_K K(x_1, \dotsc, x_n) \to \Omega_{K(x_1, \dotsc, x_n)/K} = \bigoplus dx_i \cdot K(x_1, \dotsc, x_n).$$
In order to prove Proposition~\ref{prop:crossratiodegcase}, it therefore suffices to prove the following result.
\begin{proposition}\label{prop:linearindependence}
Let $S_1, \dotsc, S_{n-3} \subset [n]$ with each $S_j$ of size $4$. If $S_1, \dotsc, S_{n-3}$ satisfies the Cerberus condition, then $d\operatorname{cr}_{S_1}, \dotsc, d\operatorname{cr}_{S_{n-3}}$ are linearly independent in $\Omega_{K(x_1, \dotsc, x_n)/K}$. 
\end{proposition}

To prove Proposition~\ref{prop:linearindependence}, we use a result from the theory of error correcting codes called the GM-MDS theorem.
Although the theorem has several variants, we present the result in the manner which is most useful for us.\footnote{For some additional context, a linear code $C$ is a $k$-dimensional subspace of $\mathbb{F}_q^n$ for some $q$. A code is maximum distance separable (MDS) if all Pl\"{u}cker coordinates of this subspace are nonzero. A \emph{generator matrix} is a matrix whose rows form a basis for $C$. A code $C$ has many possible generator matrices. The GM-MDS conjecture of \cite{DauSongYuen} gave a combinatorial description of which entries of a generator matrix can be chosen to be zero for some MDS code $C$ over $\mathbb{F}_q$. In the course of proving the GM-MDS conjecture, it was found that it is sufficient for $C$ to be a particular kind of code known as a Reed--Solomon code. This version of the result is how we formulate Theorem~\ref{thm:gm-mds}.} 

\begin{theorem}[GM-MDS Theorem, \cite{DauSongYuen}, {\cite[Theorem 1 and Eq. (2)]{YildizHassibi}}, \cite{Lovett}]\label{thm:gm-mds}
    Let $n \ge m \ge 1$ be positive integers, and let $T_1, \hdots, T_{m} \subset [n]$ be sets such that for all $I \subseteq [m]$ nonempty, we have that $\left|\bigcup_{i \in I} T_i\right| \ge n - m + |I|$. Let $x_1, \dotsc, x_n$ be indeterminates, and let $M$ be the $m \times n$ matrix with $M_{ij} = x_j^{i-1}$. Then there is $G \in \operatorname{GL}_m$ such that  $(GM)_{ij} = 0$ unless $j \in T_i$.
    
\end{theorem}

We also note the following observation of \cite{DauSongYuen}, which gives additional structure for $G$.

\begin{proposition}[{\cite[Eq.(18-21)]{DauSongYuen}}]\label{prop:gm-mds-extra}
    In the context of Theorem~\ref{thm:gm-mds}, if $|T_i| = n-m+1$ for all $i \in [m]$, then for all $i,\ell \in [m]$, we can take $G_{i,\ell}$ to be the coefficient of $t^{\ell}$ in $\prod_{j \in [n] \setminus T_i}(t - x_j)$. Furthermore, in that case $(GM)_{ij} = \prod_{k \in [n] \setminus T_i} (x_j - x_k)$.
\end{proposition}

We can now prove Proposition~\ref{prop:linearindependence}.

\begin{proof}[Proof of Proposition~\ref{prop:linearindependence}]
Using the quotient rule and the product rule, we compute
\begin{multline*}
 d \operatorname{cr}_{\{1, 2, 3, 4\}}
= \Big[(x_2 - x_3)(x_2 - x_4)(x_3 - x_4) dx_1 - (x_1 - x_3)(x_1 - x_4)(x_3 - x_4) dx_{2} \\ 
 + (x_1 - x_2)(x_1 - x_4)(x_2 - x_4) dx_3 - (x_1 - x_2)(x_1 - x_3)(x_2 - x_3) dx_4 \Big] (x_1-x_4)^{-2}(x_2-x_3)^{-2}.
\end{multline*}
It suffices to show that the $(n-3) \times n$ matrix whose rows are given by the $d\operatorname{cr}_{S_i}$ has full rank; we may check this after scaling each row to clear denominators. 
By the computation above, if $S_j = \{j_1, j_2, j_3, j_4\}$, then the $i$-th row has $(-1)^{\ell + 1} \prod_{a < b \in S_i, a, b \not= j_{\ell}} (x_a - x_b)$ in the $j_{\ell}$-th entry and is zero otherwise. 

Let $M$ be the $(n-3) \times n$ matrix whose $i$-th column is given by $(1, x_i, \dotsc, x_i^{n-4})$. We apply Theorem~\ref{thm:gm-mds} with $m = n-3$ and each $T_i = S_i$. Because $S_1, \dotsc, S_{n-3}$ satisfies the Cerberus condition, there is an invertible $(n-3) \times (n-3)$ matrix $G$ such $(GM)_{ij} = 0$ unless $j \in S_i$. By Proposition~\ref{prop:gm-mds-extra}, we may pick $G$ such that $(GM)_{ij} = \prod_{k \in [n]\setminus S_i}(x_j - x_k)$. 

If we multiply the $i$-th row of $GM$ by $\prod_{a < b \in S_i} (x_a - x_b)$ and divide the $j$-th column of $GM$ by $\prod_{k \in [n] \setminus j} (x_j - x_k)$, then the $(i, j)$-th entry is $\frac{\prod_{a < b \in S_i} (x_a - x_b)}{\prod_{k \in S_i, k \not= j} (x_j - x_{k})}$. This is exactly the matrix obtained from taking the differentials of the cross-ratios and clearing denominators. Because $G$ is invertible and $M$ has full-rank, we see that $GM$ has full rank. 
\end{proof}

\begin{remark}\label{rem:char}
Because the GM-MDS theorem holds over a field of any characteristic, Proposition~\ref{prop:linearindependence} implies that, in any characteristic, the rational functions corresponding to cross-ratios form a separating transcendence basis when the Cerberus condition is satisfied. This implies that the map $\overline{M}_{0,n} \to (\mathbb{P}^1)^{n-3}$ given by $n-3$ cross-ratios is generically \'{e}tale when it is dominant. 
\end{remark}

\begin{remark}\label{rem:comb}
It is tempting to try to prove Theorem~\ref{thm:cerberus} using the non-negative recursion given by Theorem~\ref{thm:recursion}. The following combinatorial statement is equivalent to Proposition~\ref{prop:crossratiodegcase}: Let $S_1, \dotsc, S_{n-3} \subset [n]$ be sets satisfying the Cerberus condition with each $S_j$ of size $4$, and suppose $S_1 = \{1, 2, 3, 4\}$. Then we can partition $[n]$ into two parts $P$ and $Q$ with $\{1, 2\} \subset P$ and $\{3, 4\} \subset Q$ such that
\begin{enumerate}
\item For all $j > 1$, $|S_j \cap P| \not= 2$.
\item Say $S_2, \dotsc, S_k$ have $|S_j \cap P| > 2$, and $S_{k+1}, \dotsc, S_{n-3}$ have $|S_j \cap Q| > 2$. Form a set system $T_2, \dotsc, T_{k}$ in $P \cup \star$ by setting $T_j = S_j$ if $|S_j \cap P| = 4$, and $T_j = (S_j \cap P) \cup \star$ if $|S_j \cap P| = 3$ for $j = 2, \dotsc, k$, and similarly form a set system in $Q \cup \star$ using $S_{k+1}, \dotsc, S_{n-3}$. Then both of these set systems satisfy the Cerberus condition. 
\end{enumerate}
We have been unable to prove this statement directly. In a communication with Hunter Spink we were able to prove it in the case when $1 \in S_j$ for all $j$, when there is a unique choice of bipartition that works.
\end{remark}

Let $I$ be the group of Euclidean isometries of the plane $\RR^2$.  For a graph $G$ with vertex set $[v]$ and edges $E$, consider the map $\Psi_G \colon (\RR^2)^{[v]} / I \to \RR^E$ given by $[p \colon [v] \to \RR^2] \mapsto \big(\| p(a) - p(b)\|^2\big)_{(a,b)\in E}$.
The graph $G$ is said to be \emph{generically rigid} if $\Psi_G$ has finite fibers over general points of $\operatorname{Image}(\Psi_G)$.
Note that $\dim \operatorname{Image}(\Psi_G) = 2v-3$ if $G$ is generically rigid, and thus, if $|E|>2v-3$, there is some $E'\subsetneq E$ such that $G' = ([v],E')$ is generically rigid.
Hence, a generically rigid graph $G$ is said to be \emph{minimally rigid} if $\Psi_G$ is further dominant.
By combining Theorem~\ref{thm:cerberus} with \cite{LamanSphere}, we obtain a new proof of Laman's theorem \cites{PollaczekGeiringer,Laman}, stated below.

\begin{corollary}\label{cor:laman}
A graph $G = ([v],E)$ is minimally rigid if and only if it has $2v-3$ edges and every induced subgraph on $v'>1$ vertices has at most $2v'-3$ edges.
\end{corollary}

\begin{proof}
Let $S = \{(x,y,z) \in \RR^3 : x^2 + y^2 + z^2 = 1\}$ be the 2-dimensional sphere, and consider the map $\widetilde\Psi_G \colon S^{[v]} / SO_3(\RR) \to \RR^E$ given by $[p\colon [v] \to S] \mapsto \big(\frac{1 - p(a) \bullet p(b)}{2}\big)_{(a,b)\in E}$, where $\bullet$ denotes the standard dot product, which is essentially the arc length between $p(a)$ and $p(b)$.
A direct argument \cite{Pogorelov} shows that a graph is minimally rigid (in the plane) if and only if it is \emph{minimally rigid in the sphere}, i.e., the map $\widetilde\Psi_G$ is generically finite.  By dimensional considerations, this happens only if $|E| = 2v-3$, and in this case we need check whether $\widetilde\Psi_G$ is dominant, which can be done after base changing to $\CC$.  \cite{LamanSphere} showed that the degree of the complexified map $\widetilde\Psi_{G, \CC}$ is equal to a cross-ratio degree on $\Mbar_{0,2v}$, as described in \ref{ref:lamansphere}.  Translating the Cerberus condition on subsets of $[2v]$ of the form $\{a,b,a+v, b+v\}$ for $(a,b) \in E$ to a condition on induced subgraphs of $G$ and applying Theorem~\ref{thm:cerberus} yields the desired result.
\end{proof}

\section{Recursion}\label{sec:recursion}

In this section, we prove Theorem~\ref{thm:recursion} and deduce from it an interpretation of cross-ratio degrees, i.e., Kapranov degrees when $|S_j| = 4$ for all $j$, in terms of marked trivalent trees (Corollary~\ref{cor:comb}). 

\medskip
We prepare with two lemmas about $X_{S,i}$ and boundary divisors on $\Mbar_{0,n}$.
For a partition $[n] = P\sqcup Q$ with $|P|\geq 2$ and $|Q|\geq 2$, let $D_{P|Q}$ be the corresponding boundary divisor, i.e., 
\[
D_{P|Q} = \text{the closure of }\left\{C \in \Mbar_{0,n} : \begin{matrix} \text{$C$ has two $\PP^1$-components, each with distinct marked points}\\ \text{labelled by $P$ or $Q$ respectively}\end{matrix}\right\}.
\]
With $\{\star\}$ a singleton disjoint from $[n]$, the gluing map
$\operatorname{gl}_{P|Q} \colon \Mbar_{0, P \cup \star} \times \Mbar_{0, Q \cup \star} \to \Mbar_{0,n}$, which glues two marked curves $C_1\in  \Mbar_{0, P \cup \star}$ and $C_2\in  \Mbar_{0, Q \cup \star}$ by identifying the points marked by $\star$, is an isomorphism onto the boundary divisor $D_{P|Q}$.
Let $\pi_P \colon \Mbar_{0, P \cup \star} \times \Mbar_{0, Q \cup \star} \to \Mbar_{0, P \cup \star}$ be the projection onto the first factor, and similarly for $\pi_Q$. 

\begin{lemma}\label{lem:boundary}
Let $(S,i)$ be a pair with $S\subseteq [n]$, $|S|\geq 4$, and $i\in S$. If $a \in S$ is an element distinct from $i$, then we have that
\[
X_{S, i} = X_{S\setminus a, i} + \sum_{P, Q} D_{P | Q},
\] where the sum ranges over all bipartitions $[n] = P \sqcup Q$ such that $\{a, i\} \subseteq P$ and $S \setminus \{a, i\} \subseteq Q$. 
\end{lemma}

\begin{proof}
Consider the forgetful map $f_{S\setminus a} \colon \Mbar_{0,S} \to \Mbar_{0,S\setminus a}$. By \cite[Lemma 3.1]{ArabelloCornalba}, we have that
$\psi_i = X_{S\setminus a, i} + D_{\{a, i\} | (S \setminus \{a,i\})}$ on $\overline{M}_{0, S}$. 
The pullback of any $D_{P' |Q'}$ along $f_S \colon \Mbar_{0, n} \to \Mbar_{0, S}$ is the sum of the boundary divisors on $\Mbar_{0,n}$ corresponding to the partitions obtained by adding elements of $[n] \setminus S$ to either $P'$ or $Q'$. The result then follows. 
\end{proof}

\begin{lemma}\label{lem:glueing}
Let $(S,i)$ be a pair with $S\subseteq [n]$, $|S|\geq 4$, and $i\in S$. For a partition $[n] = P\sqcup Q$ with $|P|\geq 2$ and $|Q|\geq 2$, we have
\[
\operatorname{gl}_{P | Q}^* X_{S, i} = \begin{cases}
\pi_P^* X_{S,i} & \text{if $S\subseteq P$}\\
\pi_P^* X_{(S \setminus i)\cup \star, \star} & \text{if $S\cap P = S\setminus i$}\\
\pi_P^* X_{(S\cap P)\cup \star, i} & \text{if $S\cap P\supseteq \{i,j\}$ for some $i\neq j \in S$}\\
\end{cases},
\]
and by symmetry, one obtains the remaining three cases by replacing $P$ with $Q$.
\end{lemma}

\begin{figure}[h]
\begin{tikzpicture}

\draw [black, thick] plot [smooth, tension=1] coordinates {(-1.5,-0.3) (0,-0.2) (1,0.6)};
\draw [black, thick] plot [smooth, tension=1] coordinates {(0.5,0.6) (1.5,-0.2) (3,-0.3)};

\node[above] at (-0.5,0.4) {$P$};
\node[above] at (2,0.4) {$Q$};

\node at (-1.2,-0.35) {\color{blue!50!}\ding{170}};
\node at (-0.8,-0.35) {\color{blue!50!}\ding{170}};
\filldraw (-0.4,-0.3) circle[radius=1.8pt];
\node at (0.0,-0.2) {\color{blue!90!}\ding{170}};
\node at (0.4,0.0) {\color{blue!50!}\ding{170}};

\node[above] at (0,-0.1) {$i$};

\filldraw (1.1,0.0) circle[radius=1.8pt];
\filldraw (1.5,-0.2) circle[radius=1.8pt];
\filldraw (1.9,-0.3) circle[radius=1.8pt];
\filldraw (2.3,-0.35) circle[radius=1.8pt];
\filldraw (2.7,-0.35) circle[radius=1.8pt];

\draw [black, thick] plot [smooth, tension=1] coordinates {(3.5,-0.3) (5,-0.2) (6,0.6)};
\draw [black, thick] plot [smooth, tension=1] coordinates {(5.5,0.6) (6.5,-0.2) (8,-0.3)};

\node[above] at (4.5,0.4) {$P$};
\node[above] at (7,0.4) {$Q$};

\node at (3.8,-0.35) {\color{blue!50!}\ding{170}};
\node at (4.2,-0.35) {\color{blue!50!}\ding{170}};
\filldraw (4.6,-0.3) circle[radius=1.8pt];
\filldraw (5, -0.2) circle[radius=1.8pt];
\node at (5.4, 0.0) {\color{blue!50!}\ding{170}};

\node[above] at (7.3,-0.25) {$i$};

\filldraw (6.1,0.0) circle[radius=1.8pt];
\filldraw (6.5,-0.2) circle[radius=1.8pt];
\filldraw (6.9,-0.3) circle[radius=1.8pt];
\node at (7.3,-0.35) {\color{blue!90!}\ding{170}};
\filldraw (7.7,-0.35) circle[radius=1.8pt];

\draw [black, thick] plot [smooth, tension=1] coordinates {(8.5,-0.3) (10,-0.2) (11,0.6)};
\draw [black, thick] plot [smooth, tension=1] coordinates {(10.5,0.6) (11.5,-0.2) (13,-0.3)};

\node[above] at (9.5,0.4) {$P$};
\node[above] at (12,0.4) {$Q$};

\filldraw (8.8,-0.35) circle[radius=1.8pt];
\filldraw (9.2,-0.35) circle[radius=1.8pt];
\filldraw (9.6,-0.3) circle[radius=1.8pt];
\node at (10.0, -0.2) {\color{blue!90!}\ding{170}};
\node at (10.4, 0.0) {\color{blue!50!}\ding{170}};

\node[above] at (10.0,-0.1) {$i$};
\node[above] at (10.4,0.1) {$j$};

\filldraw (11.1,0.0) circle[radius=1.8pt];
\node at (11.5,-0.2) {\color{blue!50!}\ding{170}};
\node at (11.9,-0.3) {\color{blue!50!}\ding{170}};
\filldraw (12.3,-0.35) circle[radius=1.8pt];
\filldraw (12.7,-0.35) circle[radius=1.8pt];

\end{tikzpicture}
\caption{Illustrations of rational curves in the three cases of \cref{lem:glueing}. The left rational component supports marked points indexed by $P$, the black round marked points are not indexed by $S$, and the blue hearts are marked points indexed by $S$.}
\label{fig:cases}
\end{figure}
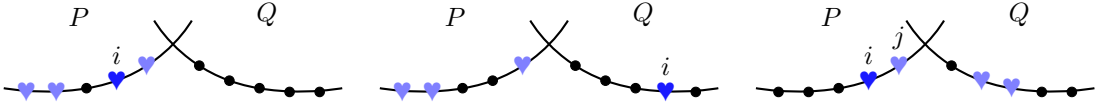

\begin{proof}
For a curve $C \in \overline{M}_{0,n}$, the curve $f_S(C)  \in \overline{M}_{0, S}$ is obtained by forgetting the points not in $S$ and contracting unstable components. The lemma follows from the fact that the divisor $X_{S, i}$ corresponds to the line bundle whose fiber at $C$ is the cotangent line at the $i$-th marked point on $f_S(C)$.
\end{proof}

\begin{proof}[Proof of Theorem~\ref{thm:recursion}]
By Lemma~\ref{lem:boundary}, we have that
\begin{equation*}\begin{split}
\int_{\Mbar_{0,n}} X_{S_1, i_1} \dotsm X_{S_{n-3}, i_{n-3}} = & \int_{\Mbar_{0,n}} X_{S_1\setminus a, i_1} X_{S_2,i_2} \dotsm X_{S_{n-3}, i_{n-3}} \\ 
&+ \sum_{\substack{[n] = P\sqcup Q\\ \{a, i_1\} \subseteq P\\ S_1 \setminus \{a,i_1\} \subseteq Q} } \int_{\Mbar_{0,n}} D_{P| Q} X_{S_2, i_2} \dotsb X_{S_{n-3}, i_{n-3}}.
\end{split}\end{equation*}
By the projection formula, 
$$\int_{\Mbar_{0,n}} D_{P| Q} X_{S_2, i_2} \dotsb X_{S_{n-3}, i_{n-3}} = \int_{\Mbar_{0, P \cup \star} \times \Mbar_{0, Q \cup \star}} \operatorname{gl}_{P | Q}^* X_{S_2, i_2} \dotsm \operatorname{gl}_{P| Q}^* X_{S_{n-3}, i_{n-3}}.$$
By Lemma~\ref{lem:glueing}, each factor $\operatorname{gl}_{P|Q}^* X_{S_j,i_j}$ on the right-hand-side is pulled back via either $\pi_P$ or $\pi_Q$. By the K\"{u}nneth formula, the intersection number is a product of intersection numbers on $\Mbar_{0, P \cup \star}$ and $\Mbar_{0, Q \cup \star}$. Then Lemma~\ref{lem:glueing} gives the result. 
\end{proof}

When $|S_j| = 4$ for all $j$, i.e., the case of cross-ratio degrees, a repeated application of Theorem~\ref{thm:recursion} yields a combinatorial interpretation of Kapranov degrees, stated in Corollary~\ref{cor:comb} below.
However, this combinatorial interpretation depends on many auxiliary choices, and we do not know whether one can use it to prove Theorem~\ref{thm:cerberus}.

\smallskip
For $T \subseteq [n]$ with $|T| = 4$, we write $X_T = X_{T, i}$ for any $i \in T$ because the class does not depend on the choice of $i$.
We say that two vertex-disjoint paths on a tree are \emph{separated by an edge} if there is a (necessarily unique) edge on the tree incident to a vertex from each path.

\begin{corollary}\label{cor:comb}
Let $S_1, \dotsc, S_{n-3}$ be subsets of $[n]$ of size $4$, and for each $j \in [n-3]$ choose a partition $S_j = \{a_j, b_j\} \sqcup \{c_j, d_j\}$. Then, the cross-ratio degree $\int_{\Mbar_{0,n}} X_{S_1} \dotsb X_{S_{n-3}}$ is the number of trivalent trees with $n$ leaves marked by $[n]$ that satisfy the following condition: The unique path from the leaf marked $a_{n-3}$ to the leaf marked $b_{n-3}$ is separated by an edge from the unique path between the leaves marked $c_{n-3}$ and $d_{n-3}$.
After contracting this edge, the unique path from $a_{n-4}$ to $b_{n-4}$ is separated by an edge from the unique path from $c_{n-4}$ to $d_{n-4}$.  After contracting this edge, similarly for $S_{n-5}$ and so forth.
\end{corollary}

\begin{proof}
By Theorem~\ref{thm:recursion}, we have that

\begin{align}
\int_{\Mbar_{0,n}} X_{S_1} \dotsm X_{S_{n-3}} &= \sum_{\substack{[n] = P\sqcup Q\\ \{a_1, b_1\} \subseteq P\\ \{c_1, d_1\} \subseteq Q}} \left( \int_{\Mbar_{0, P \cup \star}} \prod_{j=2}^{n-3} \operatorname{res}_P(S_j,a_j) \right) \cdot \left( \int_{\Mbar_{0, Q \cup \star}} \prod_{j=2}^{n-3} \operatorname{res}_Q(S_j,a_j)\right),\label{eq:4rec}
\end{align}
as $\int_{\Mbar_{0,n}} X_{S_1\setminus b_1, a_1} X_{S_2, a_2} \dotsm X_{S_{n-3}, a_{n-3}} = 0$. For $P \subset [n]$ and $j \in [n-3]$, define
\[ 
\widehat{\operatorname{res}}_P(S_j) = \begin{cases}
S_j & S_j \subseteq P\\
(S_j \cap P) \cup \{*\} & |S_j \cap P| = 3.\\
\emptyset & \text{otherwise}
\end{cases}
\]
as the combinatorial analogue of $\operatorname{res}_P(S_j, i_j)$ simplified for 4-element sets. For our family of sets $\mathcal S := \{S_j = \{a_j, b_j\} \sqcup \{c_j, d_j\} : j \in [n-3]\}$, we say that a partition $P \sqcup Q$ is \emph{good} if it satisfies the following criteria:
\begin{itemize}
  \item $a_1, b_1 \in P$ and $c_1, d_1 \in Q$
  \item For $R \in \{P, Q\}$, the collection $\mathcal S_R := \{\widehat{\operatorname{res}}_R(S_j) : j \in \{2, \hdots, n-3\}, |S_j \cap R| \ge 3\}$ has size $|R|-2$ and satisfies the Cerberus conditions on $R \cup \{*_1\}$, where we label $*$ after $S_1$ to disambiguate.
\end{itemize}
As in Remark~\ref{rem:comb}, only good partitions can give a nonzero contribution to (\ref{eq:4rec}).

By recursively unrolling the integrals in (\ref{eq:4rec}), we can identify $\int_{\Mbar_{0,n}} X_{S_1} \dotsm X_{S_{n-3}}$ with the number of \emph{good binary trees}. Each internal node of a good binary tree is labeled by aground set $D$ along with a configuration $\mathcal S$ of subsets of $D$ which satisfy the Cerberus condition. The two edges from this node are labeled $P$ and $Q$, where $P \sqcup Q$ is a good partition with respect to $\mathcal S$. The two children nodes correspond to $(P \cup \{*\}, \mathcal S_P)$ and $(Q \cup \{*\}, \mathcal S_Q)$. The root is $([n], \mathcal S)$, and the leaves are of the form $(D, \emptyset)$, where $|D|=3$. Observe also that for each $S_j \in \mathcal S$, the set $S_j$ is split into two parts of size 2 in exactly one internal node of the tree.

It thus suffices to exhibit a bijection between the set of good binary trees and the set of trivalent trees with the prescribed property. To go from a good binary tree to a trivalent tree, we first recursively construct the trivalent trees $\mathcal T_P$ and $\mathcal T_Q$ corresponding to $P \cup \{*\}$ and $Q \cup \{*\}$, respectively. We combine $\mathcal T_P$ and $\mathcal T_Q$ by identifying the edges connected to $*$ in each tree such that the two $*$'s identify with distinct vertices. We label this edge by $*_i$, where $S_i$ is the set split at this node of the tree.  At the leaves of the form $(D, \emptyset)$, where $|D|=3$, we connect the three vertices to a single internal vertex. Let $\mathcal T_{n}$ be the final trivalent tree.

To see why $\mathcal T_n$ has the prescribed property, consider an arbitrary set $S_j$ with $j \in [n-3]$ and look at the internal node of the good binary tree in which $S_i$ is split into $\{a^*_i, b^*_i\}$ and $\{c^*_i, d^*_i\}$, where the $*$'s signify that the labels may have changed (repeatedly) during the recursion. The descents of $S_i$ only split $S_j$'s with $j > i$. Thus, when the edges corresponding to these larger $S_j$'s are contracted, we have that $a^*_i$ and $b^*_i$ are each adjacent to one side of $*_i$ while $c^*_i$ and $d^*_i$ are adjacent to the other side. As we traverse upward from the node splitting $S_j$ to the root of the tree, we will connect $a^*_i$ by a path (possibly of length $0$) to $a_i$, and so forth, however none of these paths use $*_i$. Thus, the path from $a_i$ to $b_i$ and the path from $c_i$ to $d_i$ are connected by a single edge once $*_{i+1}, \hdots, *_{n-3}$ are contracted, as desired.

To invert the bijection, let $\mathcal T$ be a given trivalent tree. Label the internal edges of $\mathcal T$ as $*_{n_3}, \hdots, *_1$ based on which edge is contracted when processing $S_{n-3}, \hdots, S_1$, respectively. Replace $*_1$ with two disconnected vertices labeled $*$, each connected to one of the two internal vertices incident to $*_1$. Let $P$ and $Q$ be the subsets $[n]$ on each half of the split, with $a_1, b_1 \in P$; and let $\mathcal T_P$ and $\mathcal T_Q$ be the resulting trivalent trees. For each $j \in \{2, \hdots, n-3\}$, let $\mathcal S_P = \{\widehat{\operatorname{res}}_P(S_j) :  j \in \{2, \hdots, n-3\}, |S_j \cap R| \ge 3\}$. Note that $\widehat{\operatorname{res}}_P(S_j) \in \mathcal S_P$ if and only if $*_j$ is on the $P$-side of the tree split. From this, we can see that $\mathcal S_P$ satisfies the Cerberus condition: for any subset of set indices $J$, contract the internal edges of $\mathcal T_P$ not indexed by $J$ and delete any leaf whose label does not lie in $\bigcup_{j \in J} \widehat{\operatorname{res}}_P(S_j)$. Each internal node has degree at least $3$ and there are $|J|$ internal edges, so there are at least $|J|+3$ leaves, as desired. Therefore, $P \sqcup Q$ is a good partition and we can repeate this procedure recursively. In the base case of trivalent trees with a single internal node, we map it to the label $(D, \emptyset)$, where $D$ is the set of labels of the three leaves.
\end{proof}

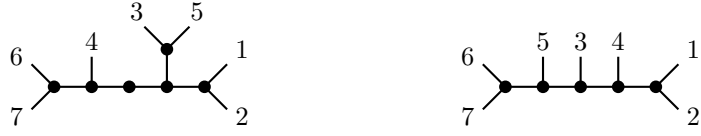
\begin{figure}[h]

    \centering
    \begin{tikzpicture}[style=thick]

     \draw (0,2) -- (2,2);
     \draw (-0.3,2.3) -- (0,2);
     \draw (-0.3,1.7) -- (0,2);
     \draw (0.5,2.4) -- (0.5,2);
     \draw (1.5,2.5) -- (1.5,2);
     \draw (1.2,2.8) -- (1.5,2.5);
     \draw (1.8,2.8) -- (1.5,2.5);
     \draw (2,2) -- (2.3, 2.3); 
     \draw (2,2) -- (2.3, 1.7);

     \foreach \x in {0,0.5,1.5,2}
     \node[inner sep=1.4pt, circle,black,draw,fill] at (\x, 2) {}; 
     \node[inner sep=1.4pt, circle,black,draw,fill] at (1.5, 2.5) {}; 
     \node at (2.5,2.5) {$1$}; 
     \node at (2.5,1.6) {$2$}; 
     \node at (1.1,3) {$3$}; 
     \node at (0.5,2.6) {$4$}; 
     \node at (1.9,3) {$5$}; 
     \node at (-0.5,2.4) {$6$}; 
     \node at (-0.5,1.6) {$7$};

     \draw (6,2) -- (8,2);
     \draw (5.7,2.3) -- (6,2);
     \draw (5.7,1.7) -- (6,2);
     \draw (6.5,2.4) -- (6.5,2);
     \draw (7,2.4) -- (7,2);
     \draw (7.5,2.4) -- (7.5,2);
     \draw (8,2) -- (8.3, 2.3); 
     \draw (8,2) -- (8.3, 1.7);

     \foreach \x in {6,6.5,7,7.5,8}
     \node[inner sep=1.4pt, circle,black,draw,fill] at (\x, 2) {}; 
     \node at (8.5,2.5) {$1$}; 
     \node at (8.5,1.6) {$2$}; 
     \node at (7,2.6) {$3$}; 
     \node at (7.5,2.6) {$4$}; 
     \node at (6.5,2.6) {$5$}; 
     \node at (5.5,2.4) {$6$}; 
     \node at (5.5,1.6) {$7$}; 
     
    \end{tikzpicture}
    \caption{The trivalent trees which, under Corollary~\ref{cor:comb}, correspond to the computation $\int_{\overline{M}_{0,7}} X_{1267} X_{4567} X_{1235} X_{1234}  = 2$.}
    \label{fig:trees}
\end{figure}

\section{Upper bound}\label{sec:upperbound}

In this section, we prove Theorem~\ref{thm:upperbound} in the following way. We may assume that $\{p, q, r\} = \{n-2, n-1, n\}$. 
For $T\subseteq [n]$ with $|T| = 4$, the forgetful map $f_T \colon \Mbar_{0,n} \to \Mbar_{0,T} \simeq \PP^1$ is given by taking the cross-ratio of the points marked by $T$.
Consider the birational map
\[
f \colon \Mbar_{0,n} \to \prod_{j=1}^{n-3} \Mbar_{0,\{j,n-2,n-1,n\}} \simeq (\mathbb{P}^1)^{n-3}.
\]
Using Poincar\'e duality on $(\PP^1)^{n-3}$ and Theorem~\ref{thm:recursion}, we will first compute the pushforward $f_* X_{S,i}$, and then deduce Theorem~\ref{thm:upperbound} by comparing the Kapranov degree with the intersection number of the pushforwards.
We start with an observation that we repeatedly use.

\begin{observation}\label{obs:divisors}
We may choose divisors $D_1, \dotsc, D_{n-3}$ in the linear series of $X_{S_1, i_1}, \dotsc, X_{S_{n-3},i_{n-3}}$, respectively, such that they intersect transversely in finitely many points, all of which are contained in the open moduli subspace ${M}_{0,n}$ of irreducible marked curves.  This is possible because the line bundles corresponding to the $X_{S,i}$ are base-point-free.
\end{observation}

We now prepare the computation of the pushforward $f_*X_{S,i}$ with a lemma.
As before, when $|T| = 4$ we write $X_T = X_{T,i}$ for any $i\in T$, which is the pullback of the hyperplane class from $\Mbar_{0,n} \to \Mbar_{0,T} \simeq \PP^1$.

\begin{lemma}\label{lem:crossratio1}
Let $S_1, \dotsc, S_{n-3} \subset [n]$ be subsets of size $4$ such that, after reordering, we have that $|\bigcup_{j=1}^{\ell} S_j| = \ell + 3$ for all $\ell =1,\dotsc, n-3$. Then 
$$\int_{\Mbar_{0,n}} X_{S_1} \dotsb X_{S_{n-3}} = 1.$$
\end{lemma}

\begin{proof}
By Observation~\ref{obs:divisors}, for generic complex numbers $\lambda_1, \dotsc, \lambda_{n-3}$, the intersection number is the number of points in ${M}_{0,n}$ such that the cross-ratio of the points in $S_j$ is $\lambda_j$ for all $j = 1,\dotsc,n-3$. There is a unique such point: we set three of the marked points in $S_1$ to be $0, 1,$ and $\infty$, and imposing the cross-ratio condition given by each $S_j$ determines an additional point.
\end{proof}

Let $T_j = \{j, n-2, n-1, n\}$ for $j \in \{1, \dotsc, n-3\}$.
The following proposition will amount to computing the pushforward $f_*X_{S,i}$.

\begin{proposition}\label{prop:pushforward}
For $S\subseteq [n]$ with $|S|\geq 4$ and $i\in S$, we have
$$\int_{\Mbar_{0, n}} X_{S, i} X_{T_2} X_{T_3} \dotsm X_{T_{n-3}} = \begin{cases}
|S|-3 & \text{if } i = 1 \\
1 & \text{if }1 \in S \setminus i \\  0 &\text{otherwise, i.e., if $1 \notin S$}.
\end{cases}$$
\end{proposition}

\begin{proof}
Let us first recall the general observation that a Kapranov degree $\int_{\Mbar_{0,n}} X_{S_1,i_1}\dotsm X_{S_{n-3},i_{n-3}}$ is zero when $\bigcup_{i=1}^{n-3} S_i \neq [n]$ because in that case $X_{S_1,i_1}\dotsm X_{S_{n-3},i_{n-3}}$ is the pullback of a product of $n-3$ divisors on $\Mbar_{0,\bigcup_i S_i}$, which has dimension at most $n-4$.

If $1\notin S$, this immediately implies that the Kapranov degree in our proposition is zero.
For the remaining cases, we first note that, when we apply Theorem~\ref{thm:recursion} with $a \in S\setminus i$, the term
\[
\bigg(\int_{\Mbar_{0,P\cup \star}}\prod_{j=2}^{n-3} \operatorname{res}_P(T_j,j) \bigg)\cdot \bigg( \int_{\Mbar_{0,Q\cup \star}}\prod_{j=2}^{n-3} \operatorname{res}_Q(T_j,j)\bigg)
\]
in the summation, for a partition $[n]= P \sqcup Q$ with $P\ni 1$, is nonzero only if $P = \{1,b\}$ for some $b\neq 1$ since $1\notin T_j$ for all $j>1$.
When $P = \{1,b\}$ for some $b\neq 1$, we find that the term evaluates to 1 as follows: $\int_{\Mbar_{0,P\cup \star}}\prod_{j=2}^{n-3} \operatorname{res}_P(T_j,j) = 1$ because $\operatorname{res}_P(T_j,j) = 1$ for all $j>1$ and $\Mbar_{0,P\cup \star}$ is a point, while  $\int_{\Mbar_{0,Q\cup \star}}\prod_{j=2}^{n-3} \operatorname{res}_Q(T_j,j) = 1$ by Lemma~\ref{lem:crossratio1} because $\operatorname{res}_Q(T_j,j) = X_{\{j,\star,n-1,n\}}$ if $b = n-2$ (and likewise if $b = n-1$ or $n$), and otherwise $\operatorname{res}_Q(T_j,j) = X_{\{j,n-2,n-1,n\}}$ for $j\neq b$ and $X_{\{\star,n-2,n-1,n\}}$ for $j=b$.

Now, if $1\in S\setminus i$, then applying Theorem~\ref{thm:recursion} with $a = 1$, we find $\int_{\Mbar_{0,n}} X_{S\setminus 1, i} X_{T_2} \dotsm X_{T_{n-3}} = 0$ by the general observation, and all the terms in the summation over the appropriate partitions $[n] = P \sqcup Q$ vanish except when $P = \{1,i\}$, which contributes 1.

If $i = 1$, we induct on $|S|$.  When $|S| = 4$, the result follows from Lemma~\ref{lem:crossratio1}.  If $|S| > 4$, when we apply Theorem~\ref{thm:recursion} with $a \in S \setminus \{1, n-2, n-1, n\}$, the only nonzero term in the summation over partitions is when $P = \{1, a\}$, which contributes 1. Now the desired result follows from the induction hypothesis on the remaining term $\int_{\Mbar_{0,n}} X_{S\setminus a, 1} X_{T_2} \dotsm X_{T_{n-3}}$.
\end{proof}

\begin{proof}[Proof of Theorem~\ref{thm:upperbound}]
Without loss of generality, let $\{p,q,r\} = \{n-2,n-1,n\}$.
Let $f \colon \Mbar_{0,n} \to (\mathbb{P}^1)^{n-3}$ be the birational map given by taking the cross-ratio of the points $\{j, n-2, n-1, n\}$ for $j \in \{1, \dotsc, n-3\}$. Let $y_1, \dotsc, y_{n-3}$ be the basis for $H^2((\mathbb{P}^1)^{n-3})$ such that $f^* y_j = X_{T_{j}}$.
Note that under the Poincar\'e pairing on the cohomology ring of $(\PP^1)^{n-3}$, the monomials $\frac{y_1 \dotsm y_{n-3}}{y_1}, \dotsc, \frac{y_1 \dotsm y_{n-3}}{y_{n-3}}$ form the dual basis. As $f$ is birational, the projection formula states that
$$\int_{\overline{M}_{0,n}} X_{S,i} \cdot \frac{X_{\{1,n-2,n-1,n\}} \dotsb X_{\{n-3,n-2,n-1,n\}}}{X_{\{i,n-2,n-1,n\}}} = \int_{(\mathbb{P}^1)^{n-3}} f_*X_{S,i} \cdot \frac{y_1 \dotsm y_{n-3}}{y_{i}}.$$
Then Poincar\'e duality on $(\PP^1)^{n-3}$ and Proposition~\ref{prop:pushforward} give that
\[
f_*X_{S, i} = \delta_{S,i} + \sum_{\substack{a\in S\setminus i \\ a\in [n-3]}} y_a \quad\text{where}\quad \delta_{S,i} = \begin{cases}
(|S|-3)y_i & \text{if $i \in [n-3]$}\\
0 & \text{if $i\in \{n-2,n-1,n\}$}.
\end{cases}
\]

Now, let $D_1, \dotsc, D_{n-3}$ be a choice of divisors as in Observation~\ref{obs:divisors}.
As $f$ is an isomorphism on ${M}_{0,n}$, we let $D_j'$ be the closure of $D_j \cap {M}_{0, n}$ in $(\mathbb{P}^1)^{n-3}$, which as a divisor class is equivalent to $f_* X_{S_j,i_j}$.
The intersection of the $D_j'$ contains the intersection of the $D_j$, but it may contain extra points in the boundary.
By \cite[Example 12.2.7(a)]{FultonIntersectionTheory}, we have that
\begin{equation*}
\begin{split}
\int_{\Mbar_{0, n}} X_{S_1, i_1} \dotsb X_{S_{n-3}, i_{n-3}} &\le \int_{(\mathbb{P}^1)^{n-3}} D_1' \dotsb D_{n-3}' \\ 
& = \int_{(\mathbb{P}^1)^{n-3}} \prod_{j=1}^{n-3} \left(\delta_{S_j,i_j} + \sum_{\substack{a \in S_j \setminus i_j\\ a \in [n-3]}} y_a \right)\\ 
&= \sum_{m \text{ a matching}} w(m),
\end{split}
\end{equation*}
where in the last step we use that 
\begin{equation*}
\int_{(\mathbb{P}^1)^{n-3}}y_{a_1} \dots y_{a_{n-3}} = \begin{cases} 1 & \{a_1, \dotsc, a_{n-3}\} = \{1, \dotsc, n-1\} \\ 0 & \text{otherwise}.\end{cases} \qedhere
\end{equation*}
\end{proof}

The upper bound in Theorem~\ref{thm:upperbound} and the upper bound obtained by using Proposition~\ref{prop:easy}(3) to reduce to intersections of $\psi$-classes are not comparable.
For instance, in the example featured in Figure~\ref{fig:trees}, the upper bound in Theorem~\ref{thm:upperbound} is tight.
The upper bound in Theorem~\ref{thm:upperbound} may not be tight for any choice of $\{p,q,r\}$ even in the case of cross-ratio degrees \cite[Section 4]{Silversmith}. By Theorem~\ref{thm:cerberus} and Hall's marriage theorem, the upper bound is tight for some choice of $\{p, q, r\}$ whenever the Kapranov degree vanishes. 

\begin{remark}
Using a similar idea, one can show that any cross-ratio degree on $\overline{M}_{0,n}$ is bounded by $2^{n-7}$ for $n \ge 13$. If $n \ge 13$, then in any collection $S_1, \dotsc, S_{n-3} \subseteq [n]$, each with $|S_i| = 4$, some element $i$ must occur in at least $4$ of the sets because the average number of times each element occurs is strictly greater than $3$. Let $f \colon \overline{M}_{0,n} \to \mathbb{P}^{n-3}$ be the Kapranov map associated to $i$. Then $f_*X_{S_j} = c_1(\mathcal{O}(1))$ if $i \in S_j$, and $f_*X_{S_j} = c_1(\mathcal{O}(2))$ if $i \not \in S_j$.  B\'{e}zout's theorem and \cite[Example 12.2.7(a)]{FultonIntersectionTheory} implies the bound. 
\end{remark}

\section{$3$-transversals}

In this section, we prove Theorem~\ref{thm:3transversal}.
Our strategy is to apply Theorem~\ref{thm:recursion} repeatedly along with Lemma~\ref{lem:dragonthree}, which involves the following close cousin of the Cerberus condition:
For subsets $S_1, \dotsc, S_{\ell}$ of a finite set $T$, we say that they satisfy the \emph{dragon marriage condition} if
\begin{equation}\label{eq:dragon-marriage}\tag{dragon marriage}
\ell  = |T|-1 \quad\text{and}\quad | \textstyle \bigcup_{i \in U} S_i| \ge |U| + 1 \quad\text{for all nonempty subsets $ U \subseteq T$}.
\end{equation}
Note that subsets $S_1, \dotsc, S_{n-3}$ of $[n-2]$ satisfy the dragon marriage condition if and only if $S_1 \cup \{n-1, n\}, \dotsc, S_{n-3} \cup \{n-1, n\}$ satisfies the Cerberus condition.

\begin{proposition}\label{prop:DM}
\
\begin{enumerate}
\item  Subsets $S_1, \dotsc, S_{\ell}$ of $T$ satisfy the dragon marriage condition if and only if, for any choice of $i \in T$, there is a bijection $m \colon [|T|-1] \to T \setminus i$ such that $m(j) \in S_j$ for all $j \in [|T|-1]$.
\item  For subsets $S_1, \dotsc, S_{n-3}$ of $[n-2]$, the Kapranov degree
\[
\int_{\Mbar_{0,n}} X_{S_1 \cup \{n-1,n\}, n} \dotsm, X_{S_{n-3}\cup \{n-1,n\},n}
\]
is 1 if they satisfy the dragon marriage condition and 0 otherwise.
\end{enumerate}
\end{proposition}

\begin{proof}
Statement (1) is \cite[Proposition 5.4]{Pos09}, which can be deduced from the Hall-Rado theorem (Proposition~\ref{prop:hallrado}).  For statement (2), Proposition~\ref{prop:easy}(3) implies that the Kapranov degree is at most $\int_{\Mbar_{0,n}} X_{[n],n}^{n-3} = \int_{\Mbar_{0,n}} \psi_n^{n-3} = 1$, so Theorem~\ref{thm:cerberus} implies the statement.
Alternatively, the statement (2) can be interpreted as \cite[Theorem 5.1]{Pos09}; see Remark~\ref{rem:simplices}.
\end{proof}

We now proceed to prove Theorem~\ref{thm:3transversal}.  Recall that for a pair $(T,j)$ consisting of $T\subseteq [n-3]$ and $j\in \{n-2, n-1, n\}$, we denoted $Y_{T,j} = X_{T \cup \{n-2,n-1,n\}, j}$.
Let $\mathcal T = \{(T_1, j_1), \dotsc, (T_{n-3},j_{n-3})\}$ be a collection of such pairs.

\begin{lemma}
\label{lem:dragonthree}
For a partition $[n] = P \sqcup Q$ with $|P|\geq 2$, $n-2 \in P$, and $\{n-1, n\} \subseteq Q$,
define a collection $\mathcal T|_P$ of subsets of $P\setminus \{n-2\}$, and a collection $\mathcal T|_Q$ of pairs $(T,j)$ where $T\subseteq Q\setminus \{n-1,n\}$ and $j\in \{n-2,n-1,n\}$, by
\begin{align*}
\mathcal T|_P &= \{T_i \cap P : j_i = n-2 \text{ and } T_i \cap P \neq \emptyset\} \quad\text{and}\\
\mathcal T|_Q &= \{(T_i\cap Q, j_i) :  j_i = n-2 \text{ and } T_i \cap P = \emptyset\} \cup \{(T_i \cap Q, j_i) : j_i\in \{n-1, n\}\}.
\end{align*}
Then, we have
\[
\int_{D_{P| Q}} \operatorname{gl}_{P|Q}^* Y_{T_1, j_1} \dotsb \operatorname{gl}_{P|Q}^* Y_{T_{n-3}, j_{n-3}} = \int_{\Mbar_{0,Q\cup \{n-2\}}} \prod_{(T,j) \in \mathcal T|_Q} Y_{T, j} 
\]
if the collection $\mathcal T|_P$ of subsets of $P\setminus \{n-2\}$ satisfies the dragon marriage condition, and 0 otherwise.
\end{lemma}

\begin{proof}
By Lemma~\ref{lem:glueing}, we have that
$$\operatorname{gl}_{P|Q}^* Y_{T_i, j_i} = \begin{cases}
\pi_P^* X_{T_i \cap P \cup \{n-2, \star\}, n-2} & \text{if } j_i = n-2 \text{ and } T_i \cap P \not= \emptyset \\
\pi_Q^* X_{T_i \cup \{n-1,n,\star\}, \star} & \text{if }  j_i = n-2 \text{ and } T_i \cap P = \emptyset \\
\pi_Q^* X_{T_i \cap Q \cup \{n-1, n,\star\}, j_i} & \text{otherwise, i.e.\ } j_i \in \{n-1, n\},\end{cases}$$
so that the intersection number is the product
\[
\Big( \int_{\Mbar_{0, P\cup \star}} \prod_{T\in \mathcal T|_P} X_{T\cup \{n-2, \star\}, n-2} \Big) \cdot \Big(\int_{\Mbar_{0,Q\cup \{n-2\}}} \prod_{(T,j) \in \mathcal T|_Q} Y_{T, j}  \Big),
\]
where for the second factor about $\Mbar_{0,Q\cup \{n-2\}}$, we relabeled $\star$ by $n-2$.
The lemma now follows from Proposition~\ref{prop:DM}(2).
\end{proof}

\begin{proof}[Proof of Theorem~\ref{thm:3transversal}]
First, we note that the theorem holds when $T_i = [n-3]$ for all $i\in [n-3]$.
In that case, it is straightforward to verify that the number of 3-transversals is the multinomial $\binom{n-3}{d_{n-2}, d_{n-1}, d_n}$, where $d_k$ denotes the number of $i\in [n-3]$ such that $j_i = k$.  On the other hand, by \cite{Witten} (as stated in \ref{ref:witten}), we have $\int_{\Mbar_{0,n}} \psi_{n-2}^{d_{n-2}}\psi_{n-1}^{d_{n-1}}\psi_{n}^{d_{n}} = \binom{n-3}{d_{n-2}, d_{n-1}, d_n}$.

We may assume that $j_1 = n-2$. Inducting on $n$, we show that 
$$\int_{\Mbar_{0,n}} Y_{T_1 \cup a, n-2} \dotsb Y_{T_{n-3}, j_{n-3}} - \int_{\Mbar_{0,n}} Y_{T_1, n-2} \dotsb Y_{T_{n-3}, j_{n-3}} \le |\{\text{new 3-transversals}\}|,$$
where a $3$-transversal $t$ of $(T_1 \cup a, n-2), \dotsc, (T_{n-3}, j_{n-3})$ is ``new'' if it is not a $3$-transversal of $(T_1, n-2), \dotsc, (T_{n-3}, j_{n-3})$. This implies the theorem, because if this inequality were ever strict then we would have 
$$\int_{\Mbar_{0,n}} Y_{U_1, n-2} \dotsb Y_{U_{n-3}, j_{n-3}} < |\{\text{3-transversals of }(U_1, n-2), \dotsc, (U_{n-3}, j_{n-3})\}|$$
for all $U_1 \supseteq T_1 \cup a$, $U_2 \supseteq T_2, \dotsc, U_{n-3} \supseteq T_{n-3}$.
This contradicts the fact that we have equality when $U_i = [n-3]$ for all $i$. 

Applying Theorem~\ref{thm:recursion}
and then Lemma~\ref{lem:dragonthree} with $\mathcal T = \{(T_2,j_2), \dotsc, (T_{n-3},j_{n-3})\}$, the induction hypothesis implies that
\[
\int_{\Mbar_{0,n}} Y_{T_1 \cup a, n-2} \dotsb Y_{T_{n-3}, j_{n-3}} - \int_{\Mbar_{0,n}} Y_{T_1, n-2} \dotsb Y_{T_{n-3}, j_{n-3}} = \sum_{P|Q} |\{\text{3-transversals of $\mathcal T|_Q$}\}|
\]
where the sum is over all partitions $[n] = P\sqcup Q$ with $\{a, n-2\} \subseteq P$ and $T_1 \cup \{n-1, n\} \subseteq Q$ such that the subsets of $P \setminus (n-2)$ in $\mathcal T|_P$ satisfy the dragon marriage condition.
Given such a partition $[n] = P \sqcup Q$ and a $3$-transversal $t$ of $\mathcal T|_Q$, extend $t$ to a $3$-transversal $\tilde{t}$ of $(T_1 \cup a, n-2), \dotsc, (T_{n-3}, j_{n-3})$ by setting $\tilde{t}(i) = n-2$ if $i \in P \setminus (n-2)$, and $\tilde{t}(i) = t(i)$ if $i \in Q \setminus \{n -1, n\}$. This is indeed a $3$-transversal:
if $m$ is a bijection $\{i : (T_i,j_i) \in \mathcal T|_Q\} \to Q\setminus \{n-1,n\}$ that witnesses the 3-transversal $t$, then we can extend $m$ to $\tilde{m} \colon [n-3] \to [n-3]$ by setting $\tilde{m}(1) = a$ and using Proposition~\ref{prop:DM}(1) to extend $\tilde m$ to the remaining indices $\{i : T_i \in \mathcal T|_P\}$.

In order to prove the claimed bound, we need to check that the $3$-transversal obtained is new, and that no such new $3$-transversal can be obtained from two different partitions $[n] = P\sqcup Q$.
For verifying this, we note the following observation.

\noindent
\textbf{Observation}: Any bijection $m \colon \{i : (T_i,j_i) \in \mathcal T|_Q\} \to Q\setminus \{n-1,n\}$ witnessing the 3-transversal $t$ of $\mathcal T|_Q$ must restrict to a bijection $\{i: j_i = n-2 \text{ and } T_i \cap P = \emptyset\} \to \{b \in Q\setminus \{n-1,n\} : {t}(b) = n-2\}$, 
since the other indices $i' \in \{i: (T_i,j_i)\in \mathcal T|_Q\} \setminus \{i: j_i = n-2 \text{ and } T_i \cap P = \emptyset\}$ have $j_{i'} \in \{n-1, n\}$.
Here, note that $i \geq 2$ throughout.

Suppose now that $\tilde{t}$ was not a new $3$-transversal, which would imply that there is a bijection $\tilde{m}'$ realizing $\tilde{t}$ with $\tilde{m}'(1) \neq a$.
Then, we have $\tilde{m}'(1) \in T_1 \subseteq Q\setminus \{n-1, n\}$ and $\tilde{t}(\tilde{m}'(1)) = n-2$.
But this contradicts the observation that $\tilde{m}'$ must restrict to a bijection $\{i : j_i = n-2 \text{ and } T_i \cap P = \emptyset\} \to \{b \in Q\setminus \{n-1,n\} : {t}(b) = n-2\}$.

Lastly, suppose now that there is some other partition $[n] = P' \sqcup Q'$ which also gives rise to $\tilde{t}$ via the above construction. Then, by the observation, a bijection $\tilde{m}$ realizing $\tilde{t}$ must restrict to give a bijection $\{i : j_i = n-2 \text{ and } T_i \cap P' = \emptyset\} \to \{b \in Q' \setminus \{n-1, n\}: \tilde{t}(b) = n-2\}$.
As $Q' \not= Q$, we thus have $\{i : j_i = n-2 \text{ and } T_i \cap P' = \emptyset\}\not= \{i : j_i = n-2 \text{ and } T_i \cap P = \emptyset\}$.
Hence, by switching $(P|Q)$ with $(P'|Q')$ if necessary, we may assume that
$$W = \{i : j_i = n-2, \enspace T_i \cap P = \emptyset, \text{ and } T_i \cap P' \not= \emptyset\}$$
is nonempty.
A bijection $\tilde{m}$ realizing $\tilde{t}$ then must restrict to a bijection
\[
W \to \{b \in Q\setminus \{n-1, n\} : \tilde{t}(b) = n-2 \text{ and } b \in P'\} =  \textstyle \bigcup_{i \in W} T_i \cap P'.
\]
In particular, we have $|\bigcup_{i \in W} T_i \cap P'| = |W|$, which contradicts that the subsets of $P'\setminus (n-2)$ in $\mathcal T|_{P'}$ must satisfy the dragon marriage condition.
\end{proof}

\begin{remark}
\label{rem:simplices}
For a subset $T\subseteq [n-2]$, let $\Delta_T$ be the simplex $\operatorname{conv}\{\be_i: i\in T\} \subset \RR^{n-2}$, and let $\nabla_T = \operatorname{conv}\{-\be_i : i\in T\}$.
Denote $\tilde{T} = T\cup \{n-1, n\}$.
When $\{n-1, n\} \subset S_j$ and $i_j \in \{n-1, n\}$ for all $j$, the Kapranov degree is the mixed volume of these simplices, i.e.,
$$\int_{\Mbar_{0,n}} X_{\tilde{T}_1, n-1} \dotsb X_{\tilde{T}_k, n-1} X_{\tilde{T}_{k+1}, n} \dotsb X_{\tilde{T}_{n-3}, n} = \text{mixed volume of }(\Delta_{T_1}, \dotsc \Delta_{T_k}, \nabla_{T_{k+1}}, \dotsc, \nabla_{T_{n-3}}).$$
To see this, let $LM_n$ be the Losev--Manin space \cite{LosevManin}, i.e., the moduli space of weighted stable rational curves \cite{Hassett} with $n$ points marked by $[n]$ with weights $(\epsilon, \dotsc, \epsilon, 1, 1)$ for small $\epsilon >0$.
It is isomorphic to the permutohedral toric variety; under the standard correspondence between polytopes and base-point-free divisors on toric varieties, the simplices $\Delta_T$ and $\nabla_T$ correspond to divisors on $LM_n$ which pullback to $X_{\tilde{T}, n-1}$ and $X_{\tilde{T}, n}$, respectively, under the birational map $\Mbar_{0,n} \to LM_n$.
See \cite{RossPsiClass} for details.

When $k = 0$ or $n-3$, these mixed volumes were computed in \cite[Theorem 5.1]{Pos09}, and when $n-2 \in T_j$ for all $j$, they were computed in \cite[Theorem B]{EFLS}, which is generalized by Theorem~\ref{thm:3transversal}. It can be deduced from \cites{HK12,BES} that when all of the simplices are of the form $\Delta_T$ for some $T$ or $\nabla_{[n-2]}$, then this mixed volume can be computed in terms of the characteristic polynomial of a certain cotransversal matroid. It then follows from \cite{CPV} that computing Kapranov degrees is \#P-hard. 
\end{remark}

\bibliographystyle{alpha}
\bibliography{matroid.bib}

\end{document}